\documentclass{amsart}
\usepackage{amssymb, amsmath,graphicx}
\newtoks\prt

\newtheorem{thm}{Theorem}[section]

\newtheorem{lemma}[thm]{Lemma}

\newtheorem{cor}[thm]{Corollary}
\newtheorem{example}[thm]{Example}

\newtheorem*{question}{Question}
\theoremstyle{definition}

\newtheorem{remark}[thm]{Remark}

\def\eqn#1$$#2$${\begin{equation}\label#1#2\end{equation}}

\def\F{\mathcal F}

\def\K{\mathcal K}

\def\P{\mathcal P}

\def\ep{\varepsilon}
\def\K{\mathcal K}
\def\en{\mathbb N}
\def\er{\mathbb R}

\def\ov{\overline}

\def \asep {\operatorname{asep}}

\def\span{\operatorname{span}}

\def \reg {\partial _{\kern1pt\text{reg}}}

\def\dd{\operatorname{d}}
\def\dh{\widehat{\operatorname{d}}}
\def\clu#1#2{\operatorname{clust}_{#1^{**}}(#2)}

\def\abs#1{\left|#1\right|}

\newcommand{\norm}[1]{\left\|#1\right\|}

\newcommand{\ca}[1]{\operatorname{ca}\left(#1\right)}
\newcommand{\sm}[1]{\operatorname{sm}\left(#1\right)}
\newcommand{\cca}[2][]{\operatorname{cca}_{#1}\left(#2\right)}
\newcommand{\tcca}[2][]{\widetilde{\operatorname{cca}}_{#1}\left(#2\right)}
\newcommand{\wu}[1]{\operatorname{wu}\left(#1\right)}
\newcommand{\swu}[1]{\operatorname{wus}\left(#1\right)}
\newcommand{\twu}[1]{\widetilde{\operatorname{wu}}\left(#1\right)}
\newcommand{\wca}[1]{\widetilde{\operatorname{ca}}\left(#1\right)}

\newcommand{\wk}[2][X]{\operatorname{wk}_{#1}\left(#2\right)}
\newcommand{\wck}[2]{\operatorname{wck}_{#1}\left(#2\right)}
\newcommand{\wscl}[1]{\overline{#1}^{w^*}}
\newcommand{\bs}[1]{\operatorname{bs}\left(#1\right)}
\newcommand{\wbs}[1]{\operatorname{wbs}\left(#1\right)}
\newcommand{\upleq}{\rotatebox{90}{$\,\leq\ $}}

\begin{document}

\title{Quantification of the Banach-Saks property}
\author{Hana Bendov\'a, Ond\v{r}ej F.K. Kalenda and Ji\v{r}\'{\i} Spurn\'y}

\address{Department of Mathematical Analysis \\
Faculty of Mathematics and Physic\\ Charles University\\
Sokolovsk\'{a} 83, 186 \ 75\\Praha 8, Czech Republic}
\email{haaanja@gmail.com}
\email{kalenda@karlin.mff.cuni.cz}
\email{spurny@karlin.mff.cuni.cz}


\thanks{Our research was supported in part by the grant GA\v{C}R P201/12/0290}

\begin{abstract}
We investigate possible quantifications of the Banach-Saks property and the weak Banach-Saks property.
We prove quantitative versions of relationships of the Banach-Saks property of a set with norm compactness and weak compactness. We further establish
a quantitative version of the characterization of the weak Banach-Saks property of a set using uniform weak convergence
and $\ell_1$-spreading models. We also study the case of the unit ball and in this case we prove a dichotomy which is an 
analogue of the James distortion theorem for $\ell_1$-spreading models.
\end{abstract}
\maketitle


\section{Introduction}

A bounded subset $A$ of a Banach space $X$ is said to be a \emph{Banach-Saks set} if each sequence in $A$ has a Ces\`aro convergent subsequence. A Banach space $X$ is said to have the \emph{Banach-Saks property} if each bounded sequence in $X$ has a Ces\`aro convergent subsequence, i.e., if its closed unit ball $B_X$ is a Banach-Saks set.

This property goes back to S.~Banach and S.~Saks who proved in \cite{BS} that, in the modern terminology, the spaces $L^p(0,1)$ for $p\in(1,+\infty)$ enjoy the Banach-Saks property. Any Banach space with the Banach-Saks property is reflexive \cite{NiWa} and there are reflexive spaces without the Banach-Saks property \cite{baer}. On the other hand, superreflexive spaces enjoy the Banach-Saks property
(S.~Kakutani showed in \cite{Kaku} that uniformly convex spaces have the Banach-Saks property and by \cite{enflo} any superreflexive space admits a uniformly convex renorming).

A localized version of the mentioned result of \cite{NiWa} says that any Banach-Saks set is relatively weakly compact (see \cite[Proposition 2.3]{LA}). This inspires the definition of the \emph{weak Banach-Saks property} -- a Banach space $X$ is said to have this property if any weakly compact subset of $X$ is a Banach-Saks set, i.e., if any weakly convergent sequence in $X$ admits a Ces\`aro convergent subsequence. There are nonreflexive spaces enjoying the weak Banach-Saks property, for example $c_0$ or $L^1(\mu)$ \cite{farnum,szlenk}.

In the present paper we investigate possibilities of quantifying the Banach-Saks property. This is inspired by a large number of recent results on quantitative versions of various theorems and properties of Banach spaces, see, e.g., \cite{AC-meas,AC-jmaa,CKS,CMR,f-krein,Gr-krein,kks-adv,bendova}. Another approach to quantification of the Banach-Saks property and some related properties was followed by A.~Kryczka in a recent series of papers \cite{kr08a,kr08s,kr12,kr13}. More precisely, in the quoted papers mainly quantitative versions of Banach-Saks and weak Banach-Saks operators are investigated.

The quantification means, roughly speaking, to replace implications between some notions by inequalities between certain quantities. Let us now introduce the basic quantities we will use.

Let $(x_k)$ be a bounded sequence in a Banach space. Following \cite{kks-adv} we set
\begin{eqnarray}
\ca{x_k}&=&\inf_{n\in\en} \sup\{\norm{x_k-x_l}\colon k,l\ge n\},\\
\wca{x_k}&=&\inf \{\ca{x_{k_n}}\colon (x_{k_n})\text{ is a subsequence of }(x_k)\}.
\end{eqnarray}
The first quantity measures how far the sequence is from being norm Cauchy. Clearly, $\ca{x_k}=0$ if and only if the sequence $(x_k)$ is norm Cauchy (hence norm convergent).

Since we are interested in Ces\`aro convergence of sequences, it is natural to define
\begin{eqnarray}
\cca{x_k}&=&\ca{\frac1k(\sum_{i=1}^k x_i)},\\
\tcca{x_k}&=&\inf\{\cca{x_{k_n})}\colon (x_{k_n})\text{ is a subsequence of }(x_k)\}.
\end{eqnarray}

Let us remark that the quantities $\cca{}$ and $\tcca{}$ behave differently than the quantities $\ca{}$ and $\wca{}$. More precisely, the quantity $\ca{}$ decreases when passing to a subsequence but it is not the case of $\cca{}$. Indeed, a subsequence of a Ces\`aro convergent sequence need not be Ces\`aro convergent, in fact, any bounded sequence is a subsequence of a Ces\`aro convergent sequence.

For a  bounded set $A$ in a Banach space $X$ we introduce the following two quantities:
\begin{eqnarray}
\bs{A}&=&\sup\{\tcca{x_k}\colon (x_k)\subset A\},\\
\wbs{A}&=&\sup\{\tcca{x_k}\colon (x_k)\subset A\text{ is  weakly convergent}\}.
\end{eqnarray}

The first one measures how far is $A$ from being Banach-Saks. Indeed, $\bs{A}=0$ if and only if $A$ is Banach-Saks by Corollary~\ref{c:kvalit} below. Further, the same statement yields that $\wbs A=0$ if and only if any weakly convergent sequence in $A$ has a Ces\`aro convergent subsequence (let us stress that the limit could be outside $A$). The sets with the latter property will be called \emph{weak Banach-Saks sets}.

\section{Preliminaries}

We use mostly a standard notation. If $X$ is a Banach space, $B_X$ denotes its closed unit ball. If $A$ is any set, we denote by $\# A$ the cardinality of $A$. (We use this notation mainly for finite sets).

We investigate, among others, quantifications of the relationship of the Banach-Saks property to compactness and weak compactness.
To formulate such results we need some quantities measuring non-compactness and weak non-compactness.
There are several ways how to do it. We will use the notation from \cite{kks-adv}. Let us recall the basic quantities.

If $A$ and $B$ are two nonempty subsets of a Banach space $X$, we set
\begin{eqnarray*}
	\dd(A,B)&=&\inf\{\|a-b\|: a\in A, b\in B\}, \\
	\dh(A,B)&=&\sup\{\dd(a,B) : a\in A\}.
\end{eqnarray*}
Hence, $\dd(A,B)$ is the ordinary distance of the sets $A$ and $B$ and $\dh(A,B)$ is the non-symmetrized Hausdorff distance (note that the Hausdorff distance of $A$ and $B$ is equal to $\max\{\dh(A,B),\dh(B,A)\}$).

Let $A$ be a bounded subset of a Banach space $X$. Then the Hausdorff measure of non-compactness
of $A$ is defined by
\begin{equation*}
	\chi(A)=\inf\{\dh(A,F): \emptyset\ne F\subset X\mbox{ finite}\} =
	\inf\{\dh(A,K): \emptyset\ne K\subset X\mbox{ compact}\}.
\end{equation*}

This is the basic measure of non-compactness. We will need one more such measure:
\begin{equation*}
 \beta(A)=\sup\{\wca{x_k} : (x_k)\mbox{ is a sequence in }A\}.
\end{equation*}
It is easy to check that for any bounded set $A$ we have
\begin{equation*}
		\chi(A)\le\beta(A)\le 2\chi(A),
\end{equation*}
thus these two measures are equivalent. (And, of course, they equal zero if and only if the respective set is relatively compact.)

An analogue of the Hausdorff measure of non-compactness for measuring weak non-compactness is
the de Blasi measure of weak non-compactness
\begin{equation*}
	\omega(A)=\inf\{\dh(A,K) : \emptyset\ne K\subset X\mbox{ is weakly compact}\}.	
\end{equation*}
Then $\omega(A)=0$ if and only if $A$ is relatively weakly compact. Indeed, the `if' part is obvious and the `only if' part follows from \cite[Lemma 1]{deblasi}.

There is another set of quantities measuring weak non-compactness. Let us mention two of them:
\begin{eqnarray*}
\wk{A}&=&\dh(\wscl{A},X),\\
\wck{X}{A}&=&\sup\{\ \dd(\clu{X}{x_k},X) : (x_k)\mbox{ is a sequence in }A\}.
\end{eqnarray*}
By $\wscl{A}$ we mean the weak$^*$ closure of $A$ in $X^{**}$ (the space $X$ is canonically embedded in $X^{**}$) and $\clu{X}{x_k}$ is the set of all weak$^*$ cluster points in $X^{**}$ of  the sequence $(x_k)$. 
It follows from \cite[Theorem 2.3]{AC-meas} that for any bounded subset $A$ of a Banach space $X$
we have
\begin{gather*}
\wck X{A}\le\wk{A}\le 2\wck X{A}, \\
\wk{A}\le \omega(A).
\end{gather*}

So, putting together these inequalities with the measures of norm non-compactness we obtain the
following diagram:
\begin{equation*}
	\begin{array}{cccccccc}
	&&\chi(A)&\le&\beta(A)&\le&2\chi(A) \\&&\upleq&&&&& \\
	&&\omega(A)&&&&& \\ &&\upleq&&&&& \\
	\wck{X}{A}&\le&\wk{A}&\le& 2\wck{X}{A}&&
\end{array}
\end{equation*}

Let us remark that the inequality $\omega(A)\le\chi(A)$ is obvious and that the quantities $\omega(\cdot)$ and $\wk{\cdot}$ are not equivalent, see \cite{tylli-cambridge,AC-meas}.

Some quantities related to the Banach-Saks property were defined and used in \cite{kr12}. Let us recall them. 
If $(x_k)$ is a bounded sequence in $X$, we define the \emph{arithmetic separation} of $(x_k)$ by 
$$\asep(x_k)=\inf\left\{\norm{\frac1{\#F}\left(\sum_{n\in F} x_n-\sum_{n\in H} x_n\right)} : 
F,H\subset \en, \#F=\#H, \max F<\min H\right\}.$$
Further, for any bounded set $A\subset X$ define
$$\begin{aligned}
\varphi(A)&=\sup\{\asep(x_k):(x_k)\mbox{ is a sequence in }A\},\\
\varphi'(A)&=\sup\{\asep(x_k):(x_k)\mbox{ is a weakly convergent sequence in }A\}.
\end{aligned}$$
The quantities $\asep$ and $\varphi$ are from \cite{kr12}, the quantity $\varphi'$ is an obvious modification.

\section{Quantitative relation to compactness and weak compactness}

Since any convergent sequence is also Ces\`aro convergent, relatively norm compact sets are Banach-Saks. Further, a set is Banach-Saks if and only if it is weakly Banach-Saks and relatively weakly compact. In this section we investigate quantitative versions of these relationships. Positive results are summed up in the following theorem.

\begin{thm}
\label{P:wbs<bs<beta}
Let $A$ be a bounded subset of a Banach space $X$. Then
\begin{equation}
\label{E:wbs<bs<beta}
\max\{\wck XA,\wbs{A}\}\le \bs{A}\le \beta(A).
\end{equation}
\end{thm}

The second inequality in the theorem quantifies the implication

\begin{center}
\textit{$A$ is relatively norm compact $\Rightarrow$ $A$ is Banach-Saks,}
\end{center}

the first one quantifies the implication

\begin{center}
\textit{$A$ is Banach-Saks $\Rightarrow$ $A$ is weakly Banach-Saks and relatively weakly compact.}
\end{center}

We point out that the latter implication cannot be quantified by using the de Blasi measure $\omega$ and that the converse implication is purely qualitative. This is illustrated by the following two examples.

\begin{example}
\label{EX:omega,bs}
There exists a separable Banach space $X$ such that
\[
\forall \ep>0\ \exists A\subset B_X\colon \bs{A}<\ep\ \&\ \omega(A)>\frac12.
\]
\end{example}

\begin{example}
\label{EX:be+ell1}
There exists a separable Banach space $X$ such that
\[
\forall \ep>0\ \exists A\subset B_X\colon \bs{A}=2\ \&\ \wbs{A}=0\ \&\ \omega(A)<\ep.
\]
\end{example}

The rest of this section will be devoted to the proofs of these results.
The proof of Example~\ref{EX:be+ell1} will be postponed to the next section since we will
make use of Theorem~\ref{P:wbs<2wu<4sm}.

\begin{proof}[Proof of Theorem~\ref{P:wbs<bs<beta}]

We start by the second inequality. It immediately follows from the following lemma which is a quantitative version of the well-known fact that any convergent sequence is Ces\`aro convergent.

\begin{lemma} Let $(x_k)$ be a bounded sequence in a Banach space $X$. Then
$$\cca{x_k}\le\ca{x_k}.$$
\end{lemma}

\begin{proof} Set $M=\sup_k\|x_k\|$ and fix any $c>\ca{x_k}$. Further, set
$y_m=\frac1m\sum_{i=1}^m x_{i}$ for $m\in\en$.

We find $n_0\in\en$ such that $\norm{x_{n}-x_{m}}<c$ for each $n,m\ge n_0$.
Let $\ep>0$ be given. We find $n_1\ge n_0$ such that $\frac{Mn_0}{n_1}<\ep$.
Then we have for $n_1\le n\le m$ inequalities
\[
\begin{aligned}
\norm{y_m-y_n}&= \norm{\frac1m\sum_{i=1}^ m x_{i}-\frac1n\sum_{j=1}^nx_{j}}
=\norm{\sum_{i=1}^m\sum_{j=1}^n\frac{x_i-x_j}{mn}}\\
&\le
\sum_{i=1}^m\sum_{j=1}^{n_0}\frac{\norm{x_i-x_j}}{mn}
+\sum_{i=1}^{n_0}\sum_{j=n_0+1}^{n}\frac{\norm{x_i-x_j}}{mn}
+\sum_{i=n_0+1}^{m}\sum_{j=n_0+1}^{n}\frac{\norm{x_i-x_j}}{mn} \\
&\le \frac{2M m n_0}{mn}+\frac{2Mn_0(n-n_0)}{mn}+ \frac{(m-n_0)(n-n_0)c}{mn}\\
&\le \frac{2M n_0}{n_1}+\frac{2Mn_0}{n_1} + c<c+4\ep.
\end{aligned}
\]
Since $\ep$ is arbitrary, $\cca{x_{k}}=\ca{y_m}\le c$. \end{proof}

Since the inequality $\bs A\ge \wbs A$ is obvious, it remains to prove $\bs A\ge \wck XA$. To do that we first prove the following lemma using an auxiliary quantity $\gamma_0$ defined by the formula

\begin{eqnarray*}\gamma_0(A)&=&\sup\{\ |\lim_m\lim_n x^*_m(x_n)|:      \\ && \qquad\qquad (x^*_m)\mbox{ is a sequence in } B_{X^*} \mbox{ weak* converging to }0,  \\ && \qquad\qquad (x_n) \mbox{ is a sequence in } A  \\ && \qquad\qquad\mbox{ and all the involved limits exist}\}.
\end{eqnarray*}

\begin{lemma}
\label{L:gamma0<bs}
Let $A$ be a bounded set in a Banach space $X$. Then
\begin{equation*}
\gamma_0(A)\le \bs{A}.
\end{equation*}
\end{lemma}

\begin{proof}
Let $\gamma_0(A)>c$. Then there exists a sequence $(x_k)$ in $A$ and a sequence $(x_j^*)$ in $B_{X^*}$ weak* converging to $0$ such that $\lim_j\lim_k x_j^*(x_k)>c$.
Without loss of generality we may assume that
\[
\forall j\in\en\colon \lim_k x_j^*(x_k)>c.
\]
Let $y_k=\frac1k(x_1+\cdots +x_k)$, $k\in\en$. Then
\[
\forall j\in\en\colon \lim_k x_j^*(y_k)=\lim_k x_j^*(x_k)>c.
\]
Let $\ep>0$ be arbitrary. Fix $k\in\en$. Since $(x^*_j)$ weak* converges to zero, there exists $j\in\en$ such that $x^*_j(y_k)<\ep$. Then we find $l>k$ such that $x_j^*(y_l)>c$. Then
\[
\norm{y_l-y_k}\ge x_j^*(y_l-y_k)>c-\ep.
\]
Hence $\cca{x_k}=\ca{y_k}\ge c$. Further, any subsequence of $(x_k)$ has the same properties, hence $\tcca{x_k}\ge c$. Therefore $\bs{A}\ge c$ and the proof is complete.
\end{proof}

Let us now complete the proof of the remaining inequality.

Assume first that $X$ is separable. Then $(B_{X^*},w^*)$ is metrizable, and thus angelic. 
By \cite[Theorem 6.1]{CKS}, $\gamma_0(A)=\wck{X}{A}$, and thus $\wck{X}{A}\le \bs A$ by the previous lemma.

Assume now that $X$ is arbitrary and $\wck{X}{A}>c$ for some $c>0$. Let $(x_k)$ be a sequence in $A$ with $\dd(\clu{X}{x_k},X)>c$.
Set $Y=\ov{\span}\{x_k\colon k\in\en\}$. Then $Y$ is a separable subspace of $X$ and $\dd(\clu{Y}{x_k},Y)>c$.

Indeed, let $y^{**}\in \clu{Y}{x_k}$ be arbitrary. Set
$x^{**}(x^*)=y^{**}(x^*|_Y)$, $x^*\in X^*$. Then $x^{**}\in\clu{X}{x_k}$ and for each $y\in Y$ we have due to the Hahn-Banach theorem
\[
\begin{aligned}
\norm{y^{**}-y}_{Y}&=\sup_{y^*\in B_{Y^*}} \abs{y^{**}(y^*)-y^*(y)}\\
&=\sup_{x^*\in B_{X^*}}\abs{(y^{**}(x^*|_Y)-x^*(y)}\\
&=\norm{x^{**}-y}_X.
\end{aligned}
\]
Therefore
$$\dd_{Y^{**}}(y^{**},Y)=\dd_{X^{**}}(x^{**},Y)\ge
\dd_{X^{**}}(x^{**},X),$$
hence
$$\dd(\clu{Y}{x_k},Y)\ge \dd(\clu{X}{x_k},X)>c.$$
Hence $\wck Y{A\cap Y}>c$, by the separable case we get
\[
\bs{A}\ge \bs{A\cap Y}\ge \wck{Y}{A\cap Y}>c,
\]
which concludes the proof of Theorem~\ref{P:wbs<bs<beta}.
\end{proof}

\begin{proof}[Proof of Example~\ref{EX:omega,bs}]
Let us fix $\alpha>0$ and consider an equivalent norm $|\cdot|_\alpha$ on $\ell_1$
given by the formula
$$|x|_\alpha=\max\{\alpha\|x\|_1,\|x\|_\infty\}$$
and let
$$X=\bigg(\oplus_{n=1}^\infty (\ell_1,|\cdot|_{1/n})\bigg)_{\ell_1}.$$
It is clear that $X$ is a separable Banach space. Further, $X$ has the Schur property as it is an $\ell_1$-sum of spaces with the Schur property
(this follows by a straightforward modification of the proof that $\ell_1$ has the Schur property, see \cite[Theorem~5.19]{fhhmpz}).

Further, let us define the following elements of $X$:
 $${x}^n_k=(0,\dots,0,\stackrel{n\text{-th}}{e_k},0,\dots),\qquad n,k\in\en,$$
where $e_k$ is the $k$-th canonical basic vector in $\ell_1$. Fix $n\in\en$ and set
$$A_n=\{{x}^n_k\colon k\in\en\}.$$
Since $\|{x}^n_k\|=1$ for  $k\in\en$, we get $A_n\subset B_X$. Further, $\|{x}^n_k-{x}^n_{k'}\|\ge 1$ whenever $k\ne k'$, so $\beta(A_n)\ge 1$ and hence $\chi(A_n)\ge \frac12$. Since $X$ has the Schur property, we get $\omega(A_n)=\chi(A_n)\ge\frac12$.

Let $(z_k)$ be an arbitrary sequence in $A_n$. If it has a constant subsequence, then $\tcca{z_k}=0$. Otherwise there is a one-to-one subsequence $(z_{k_i})$. It is clear that
$$\norm{\frac1m(z_{k_1}+\dots+z_{k_m})}=\norm{\frac1m({x}^n_1+\dots+{x}^n_m)}=\max\left\{\frac1n,\frac1m\right\} \mbox{\quad for }m\in\en,$$
hence $\cca{z_{k_i}}\le \frac2n$. So $\bs{A_n}\le\frac2n$.

This completes the proof.
\end{proof}

\section{Quantitative characterization of weak Banach-Saks sets}

It follows from the results summarized in \cite[Section 2]{LA} that the following assertions are equivalent
for a subset $A$ of a Banach space:
\begin{itemize}
	\item $A$ is a weak Banach-Saks set.
	\item Each weakly convergent sequence in $A$ admits a uniformly weakly convergent subsequence.
	\item No weakly convergent sequence in $A$ generates an $\ell_1$-spreading model.	
\end{itemize}

More precisely, in the quoted paper the authors formulate characterizations of Banach-Saks sets, adding to the other assertions the assumption that $A$ is relatively weakly compact. In this section we will prove a quantitative version of these characterizations. To formulate it, we need to introduce some natural quantities related to the above-mentioned properties:

Let $(x_k)$ be a sequence which weakly converges to some $x\in X$. This sequence is said to be \emph{uniformly weakly convergent} if
for each $\ep>0$
$$\exists n\in\en\ \forall x^*\in B_{X^*}\colon \#\{k\in\en\colon \abs{x^*(x_k-x)}>\ep\}\le n.$$
The quantity $\wu{x_k}$ is then defined to be the infimum of all $\ep>0$ satisfying this condition. Further, we set
\begin{eqnarray*}
\twu{x_k}&=&\inf\{\wu{x_{k_n}}\colon (x_{k_n})\text{ is a subsequence of }x_k\}.
\end{eqnarray*}
Finally, for a bounded set $A$ we define
\begin{eqnarray*}
\swu{A}&=&\sup\{\twu{x_k}\colon (x_k)\subset A\text{ is weakly convergent}\}.
\end{eqnarray*}

We continue by a definition related to spreading models.
Let $(x_k)$ be a bounded sequence. We say that it \emph{generates an $\ell_1$-spreading model with $\delta>0$} if
\[
\forall F\subset \en, \#F\le\min F\; \forall (\alpha_i)_{i\in F}\in\er^F\colon \norm{\sum_{i\in F}\alpha_ix_i}\ge \delta\sum_{i\in F} \abs{\alpha_i}.
\]
The sequence $(x_k)$ \emph{generates an $\ell_1$-spreading model} if it generates an $\ell_1$-spreading model with some $\delta>0$.

For a bounded set $A$ we set
\[
 \begin{aligned}
\sm{A}=\sup\{\delta>0\colon
&\exists (x_k)\subset A, x_k\overset{w}{\rightarrow} x,\\
& (x_k-x)\text{ generates an $\ell_1$-spreading model with $\delta$} \}.
\end{aligned}
\]

Now we are ready to formulate the promised quantitative characterizations.

\begin{thm}
\label{P:wbs<2wu<4sm}
Let $A$ be a bounded set in a Banach space $X$. Then
\begin{equation}\label{E:wbs<wu<sm}
\sm{A}\le \frac12\wbs{A}\le \swu{A}\le 2\sm{A}.
\end{equation}
\end{thm}

\begin{remark} For any bounded set $A\subset X$  have
\begin{equation}\label{eq:phi}
\sm{A}\le \frac12\varphi'(A)\le \wbs{A}.
\end{equation}
Indeed, it is easy to check that $\cca{x_k}\ge\frac12\asep(x_k)$ for any bounded sequence $(x_k)$. Since the quantity $\asep$ cannot decrease when we pass to a subsequence, we get $\tcca{x_k}\ge\frac12\asep(x_k)$, hence the second inequality in \eqref{eq:phi} follows. The first inequality follows from Lemma~\ref{l:phi} below.

Hence, by \eqref{eq:phi} and Theorem~\ref{P:wbs<2wu<4sm} the quantity $\varphi'$ inspired by \cite{kr12}
is equivalent to our quantities. Further, since clearly $\varphi'\le\varphi$, we get
$$\sm{A}\le\frac12\varphi'(A)\le \frac12\varphi(A)\le \bs{A}.$$
The last inequality follows from the previous paragraph. It seems not to be clear whether the quantity $\varphi$ is equivalent to $\bs{ }$ also for sets which are not relatively weakly compact.
\end{remark}

As a consequence of Theorem~\ref{P:wbs<2wu<4sm} we get that the introduced quantities $\wbs{}$ and $\bs{}$ really measure the failure
of the weak Banach-Saks (Banach-Saks, respectively) property of a set.

\begin{cor}\label{c:kvalit}
Let $A$ be a bounded set in a Banach space $X$.
\begin{itemize}
	\item If $\wbs{A}=0$, then $A$ is a weak Banach-Saks set.
	\item If $\bs{A}=0$, then $A$ is a Banach-Saks set.
\end{itemize}
\end{cor}

To prove the corollary we will use the following lemma, which also proves the inequality $\wbs{A}\le 2\swu{A}$ from Theorem~\ref{P:wbs<2wu<4sm}.

\begin{lemma}\label{L:ccawu}
Let $(x_k)$ be a sequence in a Banach space $X$ which weakly converges to some $x\in X$. Then $\cca{x_k}\le 2\wu{x_k}$.
\end{lemma}

\begin{proof} Let $M=\sup\limits_k \norm{x_k}$. Fix an arbitrary $c>\wu{x_k}$. Let $N\in\en$ be such that
\begin{equation}\label{E:N}
\forall x^*\in B_{X^*}\colon \#\{k\in\en\colon \abs{x^*(x_k-x)}>c\}\le N.
\end{equation}
Set $z_k=\frac1k(x_1+\cdots+x_k)$, $k\in\en$. Given $\ep>0$, we find $n_0\in\en$ such that $\frac{2MN}{n_0}<\ep$.

Now, for any couple of indices $n_0\le n<m$ and $x^*\in B_{X^*}$ we obtain from \eqref{E:N}
\[
\begin{aligned}
\abs{x^*(z_m-z_n)}&=\abs{x^*(z_m-x)+x^*(x-z_n)}\\
&=\abs{\left(\frac1m-\frac1n\right)\left(x^*(x_1-x)+\cdots +x^*(x_n-x)\right)\right.\\ &\qquad\left.+ \frac1m\left(x^*(x_{n+1}-x)+\cdots+x^*(x_m-x)\right)}\\
&\le \frac{2NM}{n}+c\left(n\left(\frac1n-\frac1m\right)+\frac{m-n}{m}  \right)\\
&\le \ep+c\left(2-2\frac{n}{m}   \right)\\
&\le 2c+\ep.
\end{aligned}
\]
Since $x^*\in B_{X^*}$ is arbitrary,
\[
\norm{z_m-z_n}\le 2c+\ep
\]
for each $n_0\le n<m$. Thus $\cca{x_k}=\ca{z_k}\le 2c$, which completes the proof.
\end{proof}

\begin{proof}[Proof of Corollary~\ref{c:kvalit}.]
Suppose that $\wbs{A}=0$. By Theorem~\ref{P:wbs<2wu<4sm} we deduce $\swu{A}=0$. Let $(x_k)$ be a weakly convergent sequence in $A$.
Then we can construct by induction sequences $(y_k^n)_{k=1}^\infty$ for $n\in\en$ such that
\begin{itemize}
	\item $y_k^1=x_k$, $k\in\en$;
	\item $(y_k^{n+1})$ is a subsequence of $(y_k^n)$;
	\item $\wu{y_k^{n+1}}<\frac1{n+1}$.
\end{itemize}
Consider the diagonal sequence $(z_k)=(y_k^k)$. Then $(z_k)$ is a subsequence of $(x_k)$, $\wu{z_k}=0$ and hence $\cca{z_k}=0$ (by Lemma~\ref{L:ccawu}),
so $(z_k)$ is Ces\`aro convergent. This completes the proof that $A$ is a weak Banach-Saks set.

For the second part, suppose that $\bs{A}=0$. Hence $\wbs{A}=0$, so by the first part, $A$ is a weak Banach-Saks set. Further, by Theorem~\ref{P:wbs<bs<beta} we get $\wck XA=0$, hence $A$ is relatively weakly compact. Therefore $A$ is a Banach-Saks set.
\end{proof}

We continue with the inequality $2\sm{A}\le \wbs{A}$. It follows immediately from the following lemma.

\begin{lemma}\label{L:sm} Let $(x_k)$ be a bounded sequence in a Banach space $X$ and $x\in X$. Suppose that  $(x_k-x)$ generates an $\ell_1$-spreading model with a constant $c$. Then $\tcca{x_{k^3}}\ge 2c$.
\end{lemma}

\begin{proof}
Since $\cca{x_k-x}=\cca{x_k}$, we may without loss of generality suppose that $x=0$. 
Let $M=\sup_k\|x_k\|$.
Let $(z_k)$ be any subsequence of $(x_{k^3})$. We will show that $\cca{z_k}\ge 2c$. To this end we notice that, for $F\subset \en$ satisfying $\#F\le (\min  F)^3$, we have $\norm{\sum_{i\in F}\alpha_iz_i}\ge c\sum_{i\in F}\abs{\alpha_i}$ whenever $(\alpha_i)$ are  arbitrary scalars.

Let $N\in\en$ be given. We set $n=N+N^2$ and $m=N+N^3$. Then the set $F=\{N+1,\dots, m\}$ satisfies $\#F\le (\min F)^3$, which implies
\[
\begin{aligned}
\left\|\frac1m\left(z_1\right.\right.&\left.\left.+\cdots+z_m\right)-\frac1n\left(z_1+\cdots+z_n\right)\right\|\\
&=\norm{\left(\frac1m-\frac1n\right)(z_1+\cdots+z_n)+\frac1m(z_{n+1}+\cdots +z_m)}\\
&\ge \norm{\left(\frac1m-\frac1n\right)(z_{N+1}+\cdots+z_n)+\frac1m(z_{n+1}+\cdots+z_m)}
\\&\qquad\qquad-\norm{\left(\frac1m-\frac1n\right)(z_1+\cdots+z_N)}\\
&\ge c\left(\frac1n-\frac1m\right)(n-N)+\frac{c}{m}(m-n)-MN\left(\frac1n-\frac1m\right)\\
&=c\frac{N^2(N^3-N^2)}{(N+N^3)(N+N^2)}+c\frac{N^3-N^2}{N^3+N}-M\frac{N(N^3-N^2)}{(N+N^3)(N+N^2)}.
\end{aligned}
\]
Since the last term converges to $2c$ as $N$ tends to infinity, $\cca{z_k}\ge 2c$, which completes the proof.
\end{proof}

The following lemma provides the promised proof of the first inequality in \eqref{eq:phi}. It is an easier variant of the previous lemma.

\begin{lemma}\label{l:phi} Let $(x_k)$ be a bounded sequence in a Banach space $X$ and $x\in X$. Suppose that  $(x_k-x)$ generates an $\ell_1$-spreading model with a constant strictly greater than $c$. Then there is $n\in\en$ such that
$\asep((x_{k^2})_{k\ge n})>2c$.
\end{lemma}

\begin{proof} Fix $d>c$ such that  $(x_k-x)$ generates an $\ell_1$-spreading model with the constant $d$. Let $M=\sup_k\norm{x_k}$.
Fix $n\in\en$ such that $n\ge2$ and 
$$2d-\sqrt{\frac2n}(d+2M)>2c.$$
If we show that 
$$\asep((x_{k^2})_{k\ge n})\ge 2d-\sqrt{\frac2n}(d+2M),$$
the proof will be completed. To do that fix any $m\in\en$ and indices 
$$n\le p_1<p_2<\dots<p_m<q_1<q_2<\dots< q_m.$$
If $m\le\frac12n^2$, then $2m\le n^2$ and therefore
$$\norm{\frac1m\left(\sum_{i=1}^m x_{p_i^2}-\sum_{i=1}^m x_{q_i^2}\right)}
=\norm{\frac1m\left(\sum_{i=1}^m (x_{p_i^2}-x)-\sum_{i=1}^m (x_{q_i^2}-x)\right)}
\ge d\cdot \frac1m\cdot 2m =2d.$$
Finally, suppose that $m>\frac12n^2$. Let $j\in\{1,\dots,m\}$ be the smallest number satisfying
$p_j^2\ge 2m$. Such a number exists since $p_m^2\ge m^2\ge 2m$ (as $m>\frac12n^2\ge2$). Moreover, necessarily $j\le \sqrt{2m}+1$. Indeed, if $p_i^2<2m$, then $i^2\le p_i^2<2m$ and hence $i<\sqrt{2m}$. We have
$$\begin{aligned}
\left\|\frac1m\left(\sum_{i=1}^m x_{p_i^2}\right.\right.&-\left.\left.\sum_{i=1}^m x_{q_i^2}\right)\right\|
=\norm{\frac1m\left(\sum_{i=1}^m (x_{p_i^2}-x)-\sum_{i=1}^m (x_{q_i^2}-x)\right)}\\
&\ge\norm{\frac1m\left(\sum_{i=j}^m (x_{p_i^2}-x)-\sum_{i=1}^m (x_{q_i^2}-x)\right)}-
\norm{\frac1m\sum_{i=1}^{j-1} (x_{p_i^2}-x)}\\
&\ge d\cdot \frac1m\cdot(2m-\sqrt{2m})-\frac1m\cdot 2M \cdot\sqrt{2m}
\\&=2d- \sqrt{\frac2m}(d+2M)\ge 2d- \sqrt{\frac2n}(d+2M).
\end{aligned}$$
This completes the proof.
\end{proof}

Finally we will show $\swu{A}\le 2\sm{A}$. This follows from the following lemma. The lemma essentially follows from \cite[Theorem 2.1]{LA}. However, since this theorem is not proved in \cite{LA} and we were not able to completely
recover it from the references therein and since we, moreover, need to know precise constants, we decided to give
here a complete proof of the lemma using the results of \cite{AMT,LA-T}.

\begin{lemma}\label{L:ramsey} Let $c>0$ and let $(x_k)$ be a weakly null sequence with $\twu{x_k}>c$. Then there is a subsequence of $(x_k)$ generating an $\ell_1$-spreading model with the  constant $\frac c2$. \end{lemma}

\begin{proof} Fix $\delta\in(0,1)$ such that $\twu{x_k}>c+3\delta$. We begin by showing that without loss of generality (up to passing to a subsequence)
we may suppose that for any finite set $F\subset\en$ the following holds:

\begin{equation}\label{eq:AMT}
\begin{aligned}
\Big(
\exists x^*\in B_{X^*}&\;\forall k\in F\colon x^*(x_k)\ge c+3\delta \Big)
\\ \implies &\exists y^*\in B_{X^*}
\Big(\forall k\in F: y^*(x_k)>c+2\delta\mbox{ and }\sum_{k\in \en\setminus F} \abs{y^*(x_k)}<\delta\Big).
\end{aligned}
\end{equation}

To this end we will use \cite[Lemma 2.4.7]{AMT}. The set 
$$D=\{ (x^*(x_k))_{k=1}^\infty \colon x^*\in B_{X^*} \}$$
is a convex symmetric weakly compact subset of $c_0$. (Indeed, the mapping $x^*\mapsto (x^*(x_k))_{k=1}^\infty$ is a mapping of $B_{X^*}$ onto $D$ which is continuous from the weak$^*$ topology to the weak topology of $c_0$.) Further, fix $\ep\in(0,1)$ such that
$$(1-\ep)(c+3\delta)>c+2\delta\mbox{ and }\ep(c+3\delta)<\delta.$$
Then the quoted lemma applied to the set $D$, the constant $c+3\delta$ in place of $\delta$ and $\ep$
yields an infinite set $M\subset \en$ such that for any finite set $F\subset M$ \eqref{eq:AMT} is satisfied with $M$ in place of $\en$.
Up to passing to a subsequence we may suppose that $M=\en$, therefore without loss of generality \eqref{eq:AMT} holds for any finite set $F\subset\en$.

Further, set
$$\K_0=\{ \{k\in\en\colon x^*(x_k)\ge c+3\delta\} \colon x^*\in B_{X^*} \},$$
let $\K$ be the family of those subsets of $\en$ which are contained in an element of $\K_0$. 

To complete the proof we will use the following lemma which is an easier variant of \cite[Theorem 2]{LA-T} and was suggested to us by the referee:

\begin{lemma}\label{LLL} Let $\F$ be a nonempty hereditary family of finite subsets of $\en$.
Then there is an infinite set $M\subset \en$ such that one of the following conditions is satisfied:
\begin{itemize}
	\item[(a)] There is some $d\in\en\cup\{0\}$ such that $\F\cap \P(M)=[M]^{\le d}$.
	\item[(b)] There is a strictly increasing mapping $f\colon M\to \en$ such that $\{F\subset M:\#F\le f(\min F)\}\subset \F$. 
\end{itemize}
\end{lemma}

(A family $\F$ is heredirary if $B\in\F$ whever $B\subset A$ and $A\in F$. Furhter, 
by $\P(M)$ we denote the power set of $M$, by $[M]^{\le d}$ the family of all subsets of $M$ of cardinality at most $d$,
by $[M]^{d}$ the family of all subsets of $M$ of cardinality exactly $d$. )

\begin{proof} Suppose that there is an infinite set $M\subset \en$ and $d\in\en$ such that $\F\cap[M]^{d}=\emptyset$.
Let $d_0\in\en$ be the minimal number with this property. If $d_0=1$, then $\F\cap\P(M)=\{\emptyset\}=[M]^{\le 0}$. Suppose that $d_0>1$. By the classical Ramsey theorem there is an infinite set $N\subset M$ such that either $\F\cap[N]^{d_0-1}=\emptyset$ or $[N]^{d_0-1}\subset \F$. The first possibility cannot occur due to the minimality of $d_0$.
The second one implies $\F\cap \P(N)=[N]^{\le d_0-1}$ (as $\F$ is hereditary and $\F\cap [N]^{d_0}=\emptyset$). Hence, the condition (a) is satisfied. 

Next suppose that such an infinite set $M\subset\en$ and $d\in\en$ do not exist. It means that $\F\cap[M]^d\ne\emptyset$ for each inifinite $M\subset \en$ and each $d\in\en$. Using the classical Ramsey theorem we deduce that for any infinite $M\subset\en$ and any $d\in\en$ there is an infinite set $N\subset M$ such that $[N]^d\subset\F$. Since $\F$ is hereditary,
automatically $[N]^{\le d}\subset F$. Therefore we can by induction construct a sequence $(M_n)$ with the following properties.
\begin{itemize}
	\item $\en\supset M_1\supset M_2\supset\cdots$
	\item $M_n$ is infinite for each $n\in\en$.
	\item $[M_n]^{\le n}\subset \F$ for each $n\in\en$.
\end{itemize}
Choose a strictly increasing sequence $(m_n)$ of natural numbers such that $m_n\in M_n$ for each $n\in\en$. Set $M=\{m_n:n\in\en\}$ and define $f\colon M\to \en$ by $f(m_n)=n$. Then $f$ witnesses that (b) is satisfied for $M$.
\end{proof}

Now we continue the proof of Lemma~\ref{L:ramsey}. We apply the preceding lemma to the family $\K$. We observe that the case (a) cannot occur. Indeed, suppose that $M\subset \en$ is infinite and $d\in\en$ such that $\K\cap\P(M)=[M]^{\le d}$.
Then, in particular, $\K_0\cap\P(M)\subset [M]^{\le d}$, hence
for each $x^*\in B_{X^*}$ we have
\begin{multline*}\#\{k\in M:\abs{x^*(x_k)}\ge c+3\delta\} = \#\{k\in M:{x^*(x_k)}\ge c+3\delta\}\\ + \#\{k\in M:{(-x^*)(x_k)}\ge c+3\delta\}\le 2d,\end{multline*}
hence $\twu{x_k}\le c+3\delta$, a contradiction with the choice of $\delta$.

Thus the case (b) must occur. Fix the relevant set $M$ and function $f$. Up to passing to a subsequence we may suppose that $M=\en$. Now we are going to check that the sequence $(x_{k})$ (which was made from the original one by passing twice to a subsequence) generates an $\ell_1$-spreading model with the constant $\frac c2$. To do that let $F\subset\en$ be a subset satisfying $\#F\le\min F$ and $(\alpha_i)_{i\in F}$ be any choice of scalars. Set $F^+=\{i\in F: \alpha_i>0\}$ and  $F^-=\{i\in F: \alpha_i<0\}$. Without loss of generality we may suppose that
$\sum_{i\in F^+}\alpha_i\ge \sum_{i\in F^-}(-\alpha_i)$ (otherwise we can pass to $(-\alpha_i)$). Since $\#F^+\le\#F\le \min F\le f({\min F})\le f(\min F^+)$, we get $F^+\in\K$. Therefore using \eqref{eq:AMT} we can find $x^*\in B_{X^*}$ such that
$$x^*(x_{i})>c+2\delta\mbox{ for }i\in F^+\mbox{\qquad and\qquad} \sum_{i\in\en\setminus F^+}|x^*(x_{i})|<\delta.$$
Then
$$\begin{aligned}\norm{\sum_{i\in F} \alpha_i x_{i}}&\ge \sum_{i\in F} \alpha_i x^*(x_{i}) \ge \sum_{i\in F^+} \alpha_i(c+2\delta)-\sum_{i\in F\setminus F^+}|\alpha_i||x^*(x_{i})|\\&\ge\frac12(c+2\delta)\sum_{i\in F} \abs{\alpha_i} - \delta\sum_{i\in F} \abs{\alpha_i} =\frac c2\sum_{i\in F} \abs{\alpha_i},\end{aligned}$$
which completes the argument.
\end{proof}

\begin{remark} The proof of Theorem~\ref{P:wbs<2wu<4sm} is inspired by the proof of \cite[Theorem 2.4]{LA}. More precisely, Lemma~\ref{L:ccawu} is straightforward. Lemma~\ref{L:sm} is a more elementary and more precise version of the argument in \cite[p. 2256, second paragraph]{LA}. The quoted approach would yield $\wbs{A}\ge\frac14\sm{A}$, with a little more care  $\wbs{A}\ge\frac12\sm{A}$. Our approach is more elementary, we use just the triangle inequality and not the possibility to extract a basic subsequence, and we obtain the best possible inequality. 
Finally, Lemma~\ref{L:ramsey} is a more precise version of the proof of \cite[Theorem 2.4(a)$\implies$(b)]{LA}.
\end{remark}

We finish this section by giving the last missing proof from the previous section.

\begin{proof}[Proof of Example~\ref{EX:be+ell1}]
Let $B$ be the Baernstein space, i.e., the separable reflexive space constructed in \cite{baer} and let $(b_n)$ denotes its canonical basis. Let $X=B\oplus_{\infty} \ell_1$ and $(e_n)$ be the standard basis of $\ell_1$. For $\ep\in[0,1]$ set $A_\ep=\{(b_n,\ep e_n)\colon n\in\en\}$.
Then $A_\ep\subset B_X$. Since $(b_n)$ converges weakly to zero, the set $A_0$ is weakly compact and hence $\omega(A_\ep)\le \dh(A_\ep,A_0)\le\ep$.

Fix $\ep\in(0,1]$. It is clear that the sequence $(b_n,\ep e_n)$ is equivalent to the $\ell_1$ basis and hence $A_\ep$ contains
no nontrivial weakly convergent sequences, and thus trivially $\wbs{A_\ep}=0$.
Finally, by the very definitions we get $\bs {A_\ep}\ge \bs{A_0}$. By the construction of $B$ in \cite{baer} we know that $(b_n)$ is weakly null and generates an $\ell_1$-spreading model with $\delta=1$. Thus $\bs {A_0}\ge2$ by Theorem~\ref{P:wbs<2wu<4sm} and hence $\bs{A_\ep}\ge 2$. This completes the proof.
\end{proof}

\section{Quantities applied to the unit ball}

In this section we investigate possible values of the quantities $\bs{}$ and $\wbs{}$ when applied to the unit ball of a Banach space. There are two main results in this section. The first one is a dichotomy for the quantity $\wbs{}$, the second one deals with the quantity $\bs{}$ in nonreflexive spaces.

\begin{thm}
\label{T:dichotomie}
Let $X$ be a Banach space. Then 
$$\wbs{B_X}=\begin{cases}0 & \mbox{ if $X$ has the weak Banach-Saks property},\\ 2 &
\mbox{ otherwise}.\end{cases}$$
In particular:
\begin{itemize}
	\item There is a separable reflexive Banach space $X$ with $\bs{B_X}=\wbs{B_X}=2$.
	\item There is a nonreflexive Banach space $X$ with separable dual with $\bs{B_X}=\wbs{B_X}=2$.
	\item If $X=C[0,1]$, then $\bs{B_X}=\wbs{B_X}=2$.
\end{itemize}
\end{thm}

\begin{thm}
\label{T:nr} \
\begin{enumerate}
	\item Let $X$ be a Banach space containing an isomorphic copy of $\ell_1$. Then $\bs{B_X}=2$. In particular, $\bs{B_{\ell_1}}=2$ and $\wbs{B_{\ell_1}}=0$.
	\item Let $X$ be a nonreflexive Banach space containing no isomorphic copy of $\ell_1$. Then $\bs{B_X}\in[1,2]$. In particular,
	$\bs{B_{c_0}}=1$, $\bs{B_{c}}=2$ and $\wbs{B_{c_0}}=\wbs{B_c}=0$, where $c$ denotes the space of convergent sequences equipped with the supremum norm.
\end{enumerate}
\end{thm}

The key ingredient of the proof of Theorem~\ref{T:dichotomie} is the following lemma which can be viewed as a variant of the James distortion theorem for spreading models.


\begin{lemma}
\label{l:dichotomie}
Let $X$ be a Banach space. Then $\sm{B_X}\in \{0,1\}$.
\end{lemma}

\begin{proof}
Assume that $\sm{B_X}>0$. Then there is a sequence $(x_k)$ in $B_X$ weakly converging to some $x\in B_X$ and $\delta>0$ such that the sequence  $(x_k-x)$ generates an $\ell_1$-spreading model with $\delta>0$, i.e.,
\[
\forall F\subset \en, \#F\le \min F\; \forall (\alpha_i)\colon \norm{\sum_{i\in F} \alpha_i (x_i-x)}\ge \delta\sum_{i\in F} \abs{\alpha_i}.
\] \
We will show that, for any $\omega>0$, there exists a weakly null sequence $(y_k)$ in $B_X$ which generates an $\ell_1$-spreading model with  $1-\omega$. 

The first step is to replace the sequence $(x_k-x)$ by a normalized weakly null sequence generating an $\ell_1$-spreading model. 
Since $(x_k-x)$ generates an $\ell_1$-spreading model, no subsequence is norm-convergent and hence $\inf\limits_k\norm{x_k-x}>0$. Thus,
if we set $u_k=\frac{x_k-x}{\norm{x_k-x}}$, then $(u_k)$ is a normalized weakly null sequence. Moreover, since $\|x_k-x\|\le2$ for each $k$,
the sequence $(u_k)$ generates an $\ell_1$-spreading model with $\frac\delta2$. By \cite[Proposition 1.5.4]{albiac-kalton} we can suppose (up to passing to a subsequence) that $(u_k)$ is a basic sequence. Moreover, by \cite[Theorem~6.6]{f-banach} we may assume (by passing to a further subsequence if necessary) that there exists a Banach space $(Y,\abs{\cdot})$ with a (subsymmetric) basis $(e_k)$ such that
\begin{equation}\label{E:dich}
\begin{aligned}
\forall \ep>0 &\;\forall N\in\en\; \exists m(\ep,N)\in \en\colon m(\ep,N)\le k_1<k_2<\cdots<k_N\implies \\
&\forall (\alpha_i)\colon(1-\ep)\abs{\sum_{i=1}^N \alpha_ie_i}\le \norm{\sum_{i=1}^N \alpha_i u_{k_i}}\le (1+\ep)\abs{\sum_{i=1}^N \alpha_ie_i}.
\end{aligned}
\end{equation}
Since $(u_k)$ generates an $\ell_1$-spreading model, the sequence $(e_k)$ is equivalent to the $\ell_1$-basis, so we may suppose that $Y$ is the space $\ell_1$ with an equivalent norm $\abs{\cdot}$ and $(e_k)$ is the canonical basis. Hence, if we set
\[
\beta=\inf\left\{\abs{\sum_{i=1}^\infty \alpha_ie_i}\colon \sum_{i=1}^\infty \abs{\alpha_i}=1\right\},
\]
then $\beta>0$ and
\[
\forall (\alpha_i)\in\ell_1\colon \abs{\sum_{i=1}^\infty \alpha_ie_i}\ge \beta\sum_{i=1}^\infty \abs{\alpha_i}.
\]

To complete the proof choose an arbitrary $\omega>0$. We can fix $\eta>0$ such that $\frac{1-\eta}{(1+\eta)^2}>1-\omega$.
Using the density of $c_{00}$ in $\ell_1$ we find  $n\in\en$ and scalars $(\alpha_i)_{i=1}^n$ such that $\sum_{i=1}^n \abs{\alpha_i}=1$ and $\abs{\sum_{i=1}^n \alpha_ie_i}<(1+\eta)\beta$.
Set $m_0=m(\eta,n)$. For every $k\in\en$  we set
\[
y_k=\frac{1}{(1+\eta)^2\beta}\sum_{i=1}^n\alpha_i u_{m_0+kn+i}.
\]
From \eqref{E:dich} and the choice of $(\alpha_i)_{i=1}^n$ we obtain
\[
\norm{y_k}\le \frac{1}{(1+\eta)^2\beta}(1+\eta)\abs{\sum_{i=1}^n \alpha_ie_i}\le \frac{(1+\eta)^2\beta}{(1+\eta)^2\beta}=1,
\]
hence $y_k$ are elements of $B_X$. Further, the sequence $(y_k)$ weakly converges to zero.

Let $N\in\en$ be fixed. Let $k_1<k_2<\cdots<k_N$ be indices, where $k_1n\ge m(\eta, nN)$. If $(\beta_j)_{j=1}^N$ are scalars,
then we have from \eqref{E:dich} and from the definition of $\beta$ estimates
\[
\begin{aligned}
\norm{\sum_{j=1}^n \beta_j y_{k_j}}&=\frac{1}{(1+\eta)^2\beta}\norm{\sum_{j=1}^N\beta_j\left(\sum_{i=1}^n \alpha_i u_{m_0+k_jn+i} \right)}\\
&\ge \frac{1-\eta}{(1+\eta)^2\beta}\abs{\sum_{j=1}^N \beta_j\left(\sum_{i=1}^n \alpha_i e_{(j-1)n+i}\right)}\\
&\ge \frac{1-\eta}{(1+\eta)^2}\sum_{j=1}^N \abs{\beta_j}\left(\sum_{i=1}^n\abs{\alpha_i}\right)\\
&=\frac{1-\eta}{(1+\eta)^2}\sum_{j=1}^N \abs{\beta_j} \ge (1-\omega)\sum_{j=1}^N \abs{\beta_j}.
\end{aligned}
\]
Hence we have shown the following statement for the sequence $(y_k)$:
\begin{multline*}
\forall N\in\en\colon
 \left(\frac{m(\eta,nN)}n\le k_1<k_2<\dots<k_N \right.\\\left.\implies
\forall (\beta_j)\colon \norm{\sum_{j=1}^N \beta_j y_{k_j}}\ge (1-\omega)\sum_{j=1}^N\abs{\beta_j}\right).
\end{multline*}
To finish the proof it is enough to extract a further subsequence from $(y_k)$ satisfying the definition of the  $\ell_1$-spreading model with $1-\omega$.
To this end, we choose an increasing sequence $(n_j)$ of indices satisfying $n_j\ge \frac{m(\eta,nj)}n$ and set $z_j=y_{n_j}$, $j\in\en$.
Let now $N\in\en$ and let $N\le k_1<k_2<\cdots< k_N$ be indices and $(\alpha_i)$ be scalars. Then
\[
\norm{\sum_{i=1}^N\alpha_j z_{k_j}}=\norm{\sum_{i=1}^N \alpha_j y_{n_{k_j}}}\ge (1-\omega)\sum_{i=1}^N \abs{\alpha_i},
\]
because $\frac{m(\eta,nN)}n\le n_N\le n_{k_1}<n_{k_2}<\cdots <n_{k_N}$.
\end{proof}

\begin{proof}[Proof of Theorem~\ref{T:dichotomie}.]
The equality follows from Lemma~\ref{l:dichotomie} and Theorem~\ref{P:wbs<2wu<4sm}.

The first example of a separable reflexive space without the Banach-Saks property is constructed in \cite{baer}.

As a nonreflexive space with separable dual which fails the weak Banach-Saks property one can take
the Schreier space described for example in \cite[Construction 0.2]{CaSh}. It is not reflexive since it contains a copy of $c_0$
by \cite[Proposition 0.7]{CaSh}, it has separable dual since it has an unconditional basis and does not contain a copy of $\ell_1$
\cite[Proposition 0.4(iii) and Proposition 0.5]{CaSh}. It fails the weak Banach-Saks property since the basis is weakly null and generates an $\ell_1$-spreading model due to \cite[Proposition 0.4(iv)]{CaSh}.

The space $C[0,1]$ fails the weak Banach-Saks property since it contains any separable space as a subspace. A direct proof is contained already
in \cite{schreier}.

\end{proof}

\begin{proof}[Proof of Theorem~\ref{T:nr}.]
(1) Let $(x_k)$ be a bounded sequence and $\delta>0$ such that
\begin{equation}\label{eq:ell1}
\norm{\sum_{i=1}^n\alpha_i x_i}\ge\delta\sum_{i=1}^n\abs{\alpha_i},\qquad n\in\en, \alpha_1,\dots,\alpha_n\in\er.
\end{equation}
Then $\cca{x_k}\ge 2\delta$. Indeed, let $m\ge n$. Then
$$\begin{aligned}
\norm{\frac 1m\sum_{i=1}^m x_i - \frac 1n\sum_{i=1}^n x_i}
&=\norm{\left(\frac1m-\frac1n\right)\sum_{i=1}^n x_i + \frac 1m\sum_{i=n+1}^m x_i}
\\ &\ge \delta\left(n\left(\frac1n-\frac1m\right)+\frac1m(m-n)\right)
=2\delta\left(1-\frac nm\right).
\end{aligned}$$
For any fixed $n$ the latter expression has limit $2\delta$ when $m\to+\infty$. This shows that $\cca{x_k}\ge 2\delta$.
Since any subsequence of $(x_k)$ satisfies \eqref{eq:ell1} as well, we get even $\tcca{x_k}\ge 2\delta$.

Let $X$ be a Banach space containing an isomorphic copy of $\ell_1$. By the James distortion theorem there is, given $\ep\in(0,1)$, a normalized sequence $(x_k)$ in $X$ which satisfies \eqref{eq:ell1} with the constant $1-\ep$ in place of $\delta$. It follows that $\bs{B_X}\ge \tcca{x_k}\ge 2(1-\ep)$. Since $\ep\in(0,1)$ is arbitrary, $\bs{B_X}\ge 2$. The converse inequality is obvious.

Finally, the equality $\wbs{B_{\ell_1}}=0$ follows from the Schur property of $\ell_1$.

(2) The inequality $\bs{B_X}\le 2$ is trivial. The inequality $\bs{B_X}\ge 1$ follows from Theorem~\ref{P:wbs<bs<beta} since for a nonreflexive space $X$ one has $\wck X{B_X}=1$ (this follows for example from \cite[Theorem 1]{Gr-James} and \cite[Proposition 2.2]{CKS}).

The spaces $c_0$ and $c$ are isomorphic and, moreover, they enjoy the weak Banach-Saks property by \cite{farnum}. Therefore $\wbs{B_{c_0}}=\wbs{B_c}=0$.

To show that $\bs{B_c}=2$ define a sequence $(x_k)$ in $B_c$ by the formula
$$x_k(i)=\begin{cases} 1, & 1\le i\le k,\\ -1, & i\ge k+1. \end{cases}$$
Let $(x_{k_n})$ be any subsequence of $(x_k)$. Denote by $y_n=\frac1n\sum_{j=1}^n x_{k_j}$ for $n\in\en$. Let $m<n$ be two natural numbers. Then
$$y_m(k_m+1)=-1 \mbox{\quad and\quad} y_n(k_m+1)=\frac1n(m\cdot(-1) +(n-m)\cdot 1)=1-\frac{2m}n.$$
Hence $\norm{y_m-y_n}\ge 2-\frac{2m}n$. For any fixed $m$ this value has limit $2$ for $n\to\infty$, thus $\cca{x_{k_n}}\ge 2$. It follows that $\tcca{x_k}\ge 2$, so $\bs{B_c}\ge 2$. 

Finally, it remains to show that $\bs{B_{c_0}}\le 1$. To do this let us fix a sequence $(x_k)$ in $B_{c_0}$. Up to passing to a subsequence we may suppose that the sequence $(x_k)$ pointwise converges to some $x\in B_{\ell^\infty}$. Fix an arbitrary $\ep\in(0,1)$. We will construct
increasing sequences of natural numbers $(k_n)$ and $(p_n)$ using the following inductive procedure:
\begin{itemize}
	\item $k_1=1$;
	\item $\abs{x_k(i)}<\ep$ for $k\le k_n$ and $i\ge p_n$;
	\item $\abs{x_k(i)-x(i)}<\ep$ for $i\le p_n$ and $k\ge k_{n+1}$.
\end{itemize}
It is clear that this construction can be performed, using the facts that the points $x_k$ belong to $c_0$ and that the sequence $(x_k)$ pointwise converges to $x$. Let us consider the sequence $(x_{k_n})$ and set $y_n=\frac1n(x_{k_1}+\dots+ x_{k_n})$ for $n\in\en$. Fix an arbitrary $i\in\en$.
Let $m\in\en$ be the minimal number such that $i\le p_m$. By the construction we have $\abs{x_{k_n}(i)}<\ep$ for $n<m$ and  $\abs{x_{k_n}(i)-x(i)}<\ep$ for $n>m$, thus
$$ \abs{x_{k_n}(i)-\frac{x(i)}2}<\frac12\abs{x(i)}+\ep \mbox{ for }n\in\en\setminus \{m\}.$$
It follows that for any $N\in\en$ we have
$$\begin{aligned} \abs{y_N(i)-\frac{x(i)}2}&\le \frac1N\sum_{n=1}^N\abs{x_{k_n}(i)-\frac{x(i)}2}\le \frac1N\left((N-1)\left(\frac12\abs{x(i)}+\ep\right)+\frac32\right)
\\&\le\frac12\abs{x(i)}+\ep +\frac3{2N}. \end{aligned} $$

Therefore for any $M,N\in\en$ we have
$$\norm{y_N-y_M}\le \|x\|+2\ep+\frac3{2N}+\frac3{2M},$$
so clearly $\cca{x_{n_k}}=\ca{y_k}\le \|x\|+2\ep$. Since $\ep\in(0,1)$ is arbitrary, we get $\tcca{x_k}\le \|x\|\le 1$. This completes the proof.
\end{proof}

\section{Final remarks and open problems}

It is natural to ask whether the inequalities in our results are optimal and which inequalities may become strict. 

Let us start by Theorem~\ref{P:wbs<bs<beta}. 
\begin{itemize}
  \item If $X=C[0,1]$ and $A=B_X$, then $\wbs{A}=\bs{A}=\beta(A)=2$, hence we have equalities. Indeed, $\wbs{A}=2$ by Theorem~\ref{T:dichotomie} and obviously $\beta(A)\le 2$.
	\item If $X=\ell_1$ and $A=B_X$, then $\wbs{A}=0$ by the Schur property of $\ell_1$, obviously $\wck X{A}\le 1$ (in fact, $\wck X{A}=1$ since $\ell_1$ is not reflexive) and $\bs{A}=\beta(A)=2$ by Theorem~\ref{T:nr}, hence the first inequality is strict, the second one becomes equality. 
	\item If $X=c_0$ and $A=B_X$, then $\wbs{A}=0$ by \cite{farnum}, $\wck X{A}=1$ since $c_0$ is not reflexive (in this concrete case the equality can
	be verified directly by the use of the sequence $(x_k)$ where $x_k=e_1+\dots+e_k$), $\bs{A}=1$ by Theorem~\ref{T:nr} and $\beta(A)=2$ (the constant $2$ is attained by the sequence $x_k=e_1+\dots+e_k-e_{k+1}$), hence the first inequality becomes equality and the second one is strict.  
\end{itemize}

So, it seems that Theorem~\ref{P:wbs<bs<beta} is optimal. 

Further, let us focus on Theorem~\ref{P:wbs<2wu<4sm}. The first inequality may become equality -- it is the case if $A=B_X$ by Theorem~\ref{T:dichotomie} and Lemma~\ref{l:dichotomie}. However, we know no example when the first inequality is strict. 
The second inequality may become equality, for example if $X$ is the Baernstein space from \cite{baer} and $A$ is the canonical basis of $X$.
In this case the last inequality is strict. We do not know any example when the second inequality is strict. So, we can ask the following 
question.

\begin{question} Let $X$ be a Banach space and $A\subset X$ a bounded set. Is it necessarily true that
$$\wbs{A}=2\sm{A}=2\swu{A}\quad ?$$
\end{question}

\section*{Acknowledgement}

We are grateful to the referee for suggesting to us Lemma~\ref{LLL} which substantially simplified the original proof of Lemma~\ref{L:ramsey}.

\def\cprime{$'$} \def\cprime{$'$}


\end{document}